\documentclass[options]{asl}
\usepackage{hyperref}
\title[DNR and WWKL]{Comparing DNR and WWKL}
    \author[Ambos-Spies]{Klaus Ambos-Spies}
    \revauthor{Ambos-Spies, Klaus}
    \address{Mathematisches Institut\\
      Universit\"at Heidelberg\\
      D-69120 Heidelberg, GERMANY}
    \email{ambos@math.uni-heidelberg.de}
    \author[Kjos-Hanssen]{Bj\o rn Kjos-Hanssen}
    \revauthor{Kjos-Hanssen, Bj\o rn}
    \address{Department of Mathematics\\
      University of Connecticut\\
      Storrs, CT 06269-3009, USA}
    \email{bjorn@math.uconn.edu}
    \author[Lempp]{Steffen Lempp}
    \revauthor{Lempp, Steffen}
    \address{Department of Mathematics\\
      University of Wisconsin\\
      Madison, WI 53706-1388, USA}
    \email{lempp@math.wisc.edu}
    \author[Slaman]{Theodore A. Slaman}
    \revauthor{Slaman, Theodore A.}
    \address{Department of Mathematics\\
      The University of California\\
      Berkeley, CA 94720-3840, USA}
    \email{slaman@math.berkeley.edu}
    \thanks{The second author was supported by a Marie Curie Fellowship of the
    European Community Programme ``Improving Human Potential'' under
    contract number HPMF-CT-2002-01888. The third author was partially
    supported by NSF grant DMS-0140120 and a Mercator Guest
    Professorship of the Deutsche Forschungsgemeinschaft. The fourth
    author was partially supported by NSF grant DMS-9988644 and the
    Humboldt Foundation. The authors would like to thank the anonymous referee for many useful comments.}


\theoremstyle{definition}

\theoremstyle{remark}

\newtheorem{thm}{Theorem}[section]

\newtheorem{lem}[thm]{Lemma}

\theoremstyle{definition}
\newtheorem{df}[thm]{Definition}
\theoremstyle{remark}

\newcommand{\la}{\langle}
\newcommand{\ra}{\rangle}

\renewcommand{\iff}{\leftrightarrow}

\newcommand{\mc}{\mathcal}

\newcommand{\union}{\cup}

\newcommand{\inter}{\cap}
\newcommand{\Union}{\bigcup}

\newcommand{\Inter}{\bigcap}
\renewcommand{\to}{\rightarrow}

\renewcommand{\and}{\,\&\,}

\begin{document}
\begin{abstract}
In Reverse Mathematics, the axiom system DNR, asserting the
existence of diagonally non-recursive functions, is strictly
weaker than WWKL$_0$ (weak weak K\"onig's Lemma).
\end{abstract}
\maketitle

\tableofcontents

\section{Introduction}\label{firstsec}

Reverse mathematics is a branch of proof theory which involves proving the
equivalence of mathematical theorems with certain collections of axioms over
a weaker base theory. In the form adopted by Harvey Friedman (see, e.g.,
\cite{Friedman:75}) and Stephen G. Simpson, expounded in the monograph
\cite{Simpson:99} and numerous papers, it involves formulating ``countable
mathematics'' in second-order arithmetic and proving mathematical theorems
$\varphi$ equivalent to suitable axioms (or axiom systems)~$\psi$ over a
weaker base axiom system~$T$, usually RCA$_0$. (Here, the subscript~$0$
denotes restricted induction, i.~e., RCA$_0$ does not include the full second
order induction scheme.) Since the model that we shall construct in order to
prove our main theorem does satisfy this scheme, subtleties of restricted
induction will have no bearing on the arguments in this paper.

Let~$T_1<T_2$ express that the theory~$T_2$ proves all the axioms of the
theory~$T_1$, but not conversely. Simpson points to the chain
\[\text{RCA$_0<$WKL$_0<$ACA$_0<$ATR$_0<\Pi^1_1$CA$_0$}\] as
consisting of the axiom systems that appear most frequently as $T\union\psi$.

In~\cite{Yu.Simpson:90}, Simpson and X. Yu introduced an axiom system
WWKL$_0$ and showed it to be strictly intermediate between RCA$_0$ and
WKL$_0$ as well as equivalent to some statements on Lebesgue and Borel
measure. WWKL$_0$ was further studied by Giusto and Simpson
\cite{Giusto.Simpson:00}; and by Brown, Giusto and Simpson
\cite{Brown.Giusto.Simpson:02}. Giusto and Simpson found that a certain
version of the Tietze Extension Theorem was provable in WKL$_0$ and implied
the DNR axiom. They pointed out that DNR is intermediate between RCA$_0$ and
WWKL$_0$, but left open the question whether DNR coincides with WWKL$_0$,
i.~e., has the same theorems as WWKL$_0$. Simpson conjectured that
DNR$<$WWKL$_0$. In the current paper, we confirm Simpson's conjecture.

\begin{df}\label{firstdef}\label{Delta}
If $\sigma,\tau\in\omega^{<\omega}$ then $\sigma$ is called a
\emph{substring} of $\tau$, $\sigma\subseteq\tau$, if for all~$x$ in the
domain of~$\sigma$, $\sigma(x)=\tau(x)$. The length of a string~$\sigma$ is
denoted by~$|\sigma|$. A string $\la a_1,\ldots,a_n\ra\in\omega^n$ is denoted
$(a_1,\ldots,a_n)$ when we find this more natural. The \emph{concatenation}
of $\la a_1,\ldots,a_n\ra$ by $\la a_n\ra$ on the right is denoted
$((a_1,\ldots,a_n),a_{n+1})$ or $\la a_1,\ldots,a_n\ra*\la a_{n+1}\ra=\la
a_1,\ldots,a_n\ra*a_{n+1}$. If $G\in\omega^\omega$ then $\sigma$ is a
substring of~$G$ if for all~$x$ in the domain of~$\sigma$, $\sigma(x)=G(x)$.

Given $G:\omega\to\omega$ and $n\in\omega$, we define the $n$th
column of $G$ to be the function $G_n:\omega\to\omega$ such that
for all $k\in\omega$, $G_n(k)=G(2^n(2k+1))$. On the other hand, if
for each $n\in\omega$ we are given a function
$G_n:\omega\to\omega$, then we let $\oplus_{n\in\omega}G_n$ denote
the function $G$ such that $G(2^n(2k+1))=G_n(k)$ for all
$n,k\in\omega$.

Let $\Phi_n$, $n\in\omega$, be a standard list of the Turing functionals. So
if $A$ is recursive in~$B$ then for some $n$, $A=\Phi_n^B$. For convenience,
if $\Phi$ is a Turing functional and for all $B$ and $x$, the computation of
$\Phi^B(x)$ is independent of~$x$, we sometimes write $\Phi^B$ instead of
$\Phi^B(x)$. Let $\Phi_{n,t}$ be the modification of $\Phi_n$ which goes into
an infinite loop after~$t$ computation steps if the computation has not ended
after~$t$ steps. We abbreviate $\Phi_n^\emptyset$ by $\Phi_n$. If the
computation $\Phi_e(x)$ terminates we write $\Phi_e(x)\downarrow$, otherwise
$\Phi_e(x)\uparrow$.

The axiom system DNR corresponds to a class of functions in
$\omega^\omega$ denoted by \textsf{DNR}: Given functions
$H,G:\omega\to\omega$, we say $H$ is $\textsf{DNR}^G$
(\emph{diagonally nonrecursive in~$G$}) if for all $x\in\omega$,
$H(x)\ne\Phi^G_x(x)$ or $\Phi^G_x(x)\uparrow$. Given
$h:\omega\to\omega$, we say $H$ is $h$-$\textsf{DNR}^G$ if in
addition for all~$n$, $H(n)<h(n)$. (This necessitates that
$h(n)>0$ for all~$n$.) We say $H$ is \textsf{DNR} if $H$ is
$\textsf{DNR}^\emptyset$. If $H$ is $\textsf{DNR}^G$ and $\sigma$
is a substring of~$H$ then $\sigma$ is called a $\textsf{DNR}^G$
\emph{string}. In this article $G:\omega\to\omega$ will be called
\emph{relatively} \textsf{DNR} if there are no $x,y$ such that
$G_{2y}(x)=\Phi^{G_0\oplus\cdots\oplus G_{2y-1}}_x(x)$.
\end{df}

\begin{df}\label{svensk}
Let $A$ be a real, i.~e., a subset of the nonnegative
integers~$\omega$. A \emph{Martin-L\"of test $U$ relative to~$A$}
is a sequence of open sets $U_n\subseteq 2^\omega$, $n\in\omega$
uniformly r.e. in~$A$ such that $\mu(U_n) \le 2^{-n}$, where $\mu$
denotes the standard measure on $2^\omega$. Then $\Inter_n U_n$ is
called a \emph{Martin-L\"of null set relative to~$A$}. If
$A=\emptyset$ then we speak simply of a \emph{Martin-L\"of test}
and a \emph{Martin-L\"of null set}. A set $R\subseteq\omega$ is
\emph{Martin-L\"of random} if for each Martin-L\"of test~$U$,
there is an~$n$ such that $R\not\in U_n$.
\end{df}

For an introduction to Martin-L\"of randomness and related
concepts the reader may consult \cite{Ambos-Spies.Kucera:00}.

The only fact we need about the axiom systems is the following:

\begin{lem}\label{homer}
Let $\mc I$ be a Turing ideal, i.~e., a set of subsets of~$\omega$ whose
Turing degrees form an ideal within the upper semilattice of all Turing
degrees. Let $N(\mc I)$ be the $\omega$-model of RCA$_0$ with $\mc I$ as the
interpretation of the power set symbol.

\noindent (1) $N(\mc I)\models\text{DNR}$ if and only if for each
$G\in\mc I$, there is $H\in\mc I$ such that $H$ is
$\textsf{DNR}^G$.

\noindent (2) $N(\mc I)\models\text{WWKL}_0$ if and only if for
each $G\in\mc I$, there is $H\in\mc I$ such that $H$ is
Martin-L\"of random relative to~$G$.
\end{lem}
\begin{proof}
For definitions of the DNR and WWKL$_0$ axioms, see
\cite{Giusto.Simpson:00}. The equivalence (1) is immediate from
the definition of the DNR axiom.

The ``if'' part of (2) follows from the relativization of a result
of Martin-L\"of~\cite{MartinLof:70}: there is a Martin-L\"of test
$(U_n)_{n\in\omega}$ (as in Definition \ref{svensk}) such that the
complement of any~$U_n$ is a $\Pi^0_1$ class of positive measure
containing only Martin-L\"of random sets. Namely,
let~$(U_n)_{n\in\omega}$ be a universal Martin-L\"of test.

The ``only if'' part of (2) follows from the relativization of a
result of Ku\v{c}era~\cite{Kucera:84}: for every Martin-L\"of
random set~$R$, every $\Pi^0_1$ class of positive measure contains
some finite modification of~$R$.
\end{proof}

We will prove the following theorem by elaborating on the proof of
Proposition 3 of~\cite{Jockusch:89}. The proof given there is attributed to
Kurtz; the result follows also from a theorem of Ku\v{c}era~\cite{Kucera:84}.

\begin{thm}\label{kurt}
There is a recursive function~$h$ such that for each Martin-L\"of random real
$R$, there is an $h$-\textsf{DNR} function~$f$ recursive in~$R$.
\end{thm}

\begin{proof}
Given any real $A\subseteq\omega$ and $x\in\omega$, let $f^*_A(x)$ be equal
to~$A$ restricted to~$x$, considered as a number $<2^x$. Let $h(x)=2^x$ and
let
$$
U_n=\{A:\exists x>n.f^*_A(x) = \Phi_x(x)\}.
$$
Then the sets~$U_n$ define a Martin-L\"of test~$U$. Hence no Martin-L\"of
random set is in all of the~$U_n$. So $f^*_R(x)=\Phi_x(x)$ for at most
finitely many~$x$. Let~$f$ be a finite modification of~$f^*$ such that~$f$ is
$h$-\textsf{DNR}. Since $R$ computes $f^*$, $R$ computes~$f$.
\end{proof}

The following theorem is proved in Section~\ref{proof}.

\begin{thm}\label{one} For any recursive function $h:\omega\to\omega$,
there exists $G:\omega\to\omega$ which is relatively \textsf{DNR}, such that
for each Turing functional~$\Phi$ and each $i\in\omega$,
$\Phi^{G_0\oplus\cdots\oplus G_i}$ is not an $h$-\textsf{DNR} function.
\end{thm}

\begin{lem}\label{lett}
Let $h$ be as in Theorem~\ref{kurt} and let $\mc I$ be the Turing ideal
generated by the functions $G_i$ (for $i\in\omega$) of Theorem~\ref{one} for
this~$h$. Then for each element~$H$ of $\mc I$, there is an element~$K$ of
$\mc I$ such that $K$ is \textsf{DNR}$^H$.
\end{lem}

\begin{proof}
Since $H$ is in $\mc I$, there exist~$y$ and~$e$ such that
$H=\Phi_e^{G_0\oplus\cdots\oplus G_{2y-1}}$. Let $K=G_{2y}$. Since $G$ is
relatively \textsf{DNR}, the proof is complete.
\end{proof}

\begin{thm}\label{DNRWWKL}
DNR is strictly weaker than WWKL$_0$.
\end{thm}

\begin{proof}
Let $h$ be as in Theorem~\ref{kurt} and let $\mc I$ be the Turing ideal
generated by the functions $G_i$ (for $i\in\omega$) of Theorem~\ref{one} for
this~$h$. By Theorem~\ref{kurt}, $\mc I$ contains no Martin-L\"of random
real. By Lemma~\ref{lett}, for each element $H$ of~$\mc I$, there is an
element~$K$ of~$\mc I$ such that~$K$ is \textsf{DNR}$^H$. Hence, by
Lemma~\ref{homer}, the $\omega$-model of RCA$_0$ whose second-order part
consists of all the sets in~$\mc I$ is a model of DNR in which WWKL$_0$ is
false.
\end{proof}

The following two theorems will not be used for the proof of
Theorem~\ref{DNRWWKL}, but seem to have independent interest.
Their proofs are based on the proof of Theorem \ref{theone-basic}.

\begin{thm}\label{wontprove}
There exists $G:\omega\to\omega$ such that $G$ is \textsf{DNR},
but $G$ does not compute any $h$-\textsf{DNR} function for any
recursive function~$h$.
\end{thm}

\begin{thm}\label{wontproveeither}
For each recursive function $h:\omega\to\omega$ there exists a
recursive function $h^*:\omega\to\omega$ and a function
$G:\omega\to\omega$ such that $G$ is $h^*$-\textsf{DNR}, but for
all Turing functionals~$\Phi$, $\Phi^G$ is not $h$-\textsf{DNR}.
In fact, $h^*$ may be chosen elementary recursive relative to $h$.
\end{thm}

Throughout the rest of this article, fix a recursive function
$h:\omega\to\omega$.

\section{Warm-up}\label{proof-basic}

This section is devoted to the proof of Theorem~\ref{theone-basic}, which
serves as a warm-up exercise for Theorem~\ref{one}.

\begin{thm}\label{theone-basic}
There exists $G:\omega\to\omega$ such that $G$ is \textsf{DNR}, but for all
Turing functionals~$\Phi$, $\Phi^G$ is not $h$-\textsf{DNR}.
\end{thm}

To satisfy the requirement that $G$ be \textsf{DNR}, it will be convenient to
use the following definition.

\begin{df}[Section~\ref{proof-basic} only]\label{phizero-basic}
Let $\Phi_0$ be a Turing functional such that for all $G:\omega\to\omega$,
$\Phi^G_0\downarrow\iff\exists x.G(x)=\Phi_x(x)$, and if
$\Phi^G_0\downarrow=i\in\omega$ then $i=0$. Let $\Phi_n$, $n\ge 1$, be the
Turing functionals of Definition~\ref{firstdef}.
\end{df}

The following definition is based on concepts in Kumabe's unpublished
preprint \cite{Kumabe:XX}, in which he establishes the existence of a
fixed-point free minimal degree.

\begin{df}[Good trees]\label{tjo-basic}
A finite set of incomparable strings in $\omega^{<\omega}$ is called a
\emph{tree}. (Note that this differs from some common notions of tree.) Given
$a\in\omega$, a nonempty tree~$T$ is called \emph{$a$-good from
$\sigma\in\omega^{<\omega}$} if
\begin{itemize}
\item[(1)] every string $\tau\in T$ extends~$\sigma$, and
\item[(2)] for each $\tau\in\omega^{<\omega}$, if there exists
$\rho\in T$ with $\sigma\subseteq\tau\subset\rho$, then there are
at least $a$ many immediate successors of~$\tau$ which are
substrings of elements of~$T$.
\end{itemize}
If $T$ is $a$-good from~$\sigma$ and $T\subseteq P\subseteq\omega^{<\omega}$,
then $T$ is called \emph{$a$-good from~$\sigma$ for~$P$}.
\end{df}

\begin{lem}\label{verygoodisgood-basic}
Let $b\ge a\ge 1$, let $T,P\subseteq\omega^{<\omega}$ and
$\sigma\in\omega^{<\omega}$. If $T$ is $b$-good from $\sigma$ for~$P$ then
$T$ is $a$-good from~$\sigma$ for~$P$.
\end{lem}

Lemma~\ref{verygoodisgood-basic} is immediate from Definition
\ref{tjo-basic}. Note, however, that a tree that contains an $a$-good tree is
not necessarily itself $a$-good.

\begin{lem}[Lemma {\rm 2.2(v)} of\/ {\rm~\cite{Kumabe:XX}}]
\label{kum-basic} Let $n\ge 1$. Given a tree~$T$ that is
$(2n-1)$-good from a string~$\alpha$ and given a set $P\subseteq
T$, there is a subset~$S$ of~$T$ which is $n$-good for~$P$ or for
$T-P$.
\end{lem}

\begin{proof}
Give the elements of~$T$ the label 1 (0) if they are in~$P$ (not
in~$P$, respectively). Inductively, suppose $\beta$ extends
$\alpha$ and is a proper substring of an element of~$T$. Suppose
all the immediate successors of $\beta$ that are substrings of
elements of~$T$ have received a label. Give $\beta$ the label 1 if
at least half of its labelled immediate successors are labelled 1;
otherwise, give $\beta$ the label 0. This process ends after
finitely many steps when $\alpha$ is given some label
$i\in\{0,1\}$. Let $S$ be the set of $i$-labelled strings in~$T$.
If $i=1$ then $S$ is contained in~$P$, and if $i=0$ then $S$ is
contained in~$T-P$, so it only remains to show that $S$ is
$n$-good.

Let $L$ be the set of all labelled strings. Note that $L$ is the
set of strings extending $\alpha$ that are substrings of elements
of $T$. For any $\beta\in L-T$, let $k$ be the number of immediate
successors of $\beta$ that are in $L$. Since $T$ is $(2n-1)$-good,
$k\ge 2n-1$. Let $p\le k$ be the number of immediate successors of
$\beta$ that have the same label as $\beta$. By construction,
$p\ge k/2$, and hence $p\ge n$. It follows that $S$ is $n$-good.
\end{proof}

The following lemma is not particularly sharp, but is sufficient for our
purposes.

\begin{lem}\label{notsharp-basic}
Let $a,n\ge 1$. Let $T$ be a tree which is $2^{a-1}n$-good from a string
$\alpha$, and let $P_1,\ldots,P_a$ be sets of strings such that
$T\subseteq\Union_i P_i$. Then for some~$i$, $T$~has a subset which is
$n$-good from $\alpha$ for~$P_i$.
\end{lem}

\begin{proof}
The case $a=1$ is trivial; the subset is $T$ itself. So assume $a\ge 2$ and
assume that Lemma~\ref{notsharp-basic} holds with $a-1$ in place of~$a$. By
Lemma~\ref{kum-basic}, if there is no $2^{a-2}n$-good subset of~$T$ from
$\alpha$ for~$P_1$ then there is a $2^{a-2}n$-good subset~$S$ of~$T$ from
$\alpha$ for the complement $\overline P_1$. As $T\inter\overline
P_1\subseteq P_2\union\cdots\union P_a$, it follows that $S$ is
$2^{a-2}n$-good from~$\alpha$ for $P_2\union\cdots\union P_a$. By Lemma
\ref{notsharp-basic} with $a-1$ in place of~$a$, $S$ has a subset~$R$ which
is $n$-good from~$\alpha$ for some~$P_i$, $i\ge 2$. As $R$ is also a subset
of~$T$, the proof is complete.
\end{proof}

\begin{df}\label{gvec-basic}
Let $\epsilon:\omega\to\omega$ be a finite partial function and write
$e_t=\epsilon(t)$ for each~$t$ in the domain of~$\epsilon$.

Let $\Phi$ be any Turing functional such that for all $G:\omega\to\omega$,
\[\Phi^G(\epsilon)\downarrow\iff\exists t\in
\text{dom}(\epsilon)\,[\Phi^G_t(e_t)\downarrow<h(e_t)]\] Given $n\in\omega$
and~$\epsilon$, let $g(n,\epsilon)=2^a n$ where
$$
a=\sum_{t\in \text{dom}(\epsilon)} h(e_t).
$$
\end{df}

Suppose we have a sequence of computations (namely, $\Phi_t(e_t)$ for those
$t$ where $e_t$ is defined) that we would like to maintain the divergence of,
while specifying more and more of the oracle for the computations. Then we
can use Definition~\ref{gvec-basic} as follows: Given $n\in\omega$, there
exists a number $g=g(n,\epsilon)\in\omega$ such that if none of the
computations $\Phi^G_t(e_t)$ converge and take values dominated by~$h$ on any
$n$-good tree of strings, then $\Phi^G(\epsilon)$ does not converge on any
$g$-good tree of strings. Lemma~\ref{g-basic} spells this out.

\begin{lem}\label{g-basic}
Let $n\ge 1$, let $\epsilon$ be a finite partial function from~$\omega$ to
$\omega$, and let $g$ be the function defined in Definition~\ref{gvec-basic}.

For each pair $(t,i)$ satisfying $i<h(e_t)$ (where $h$ is as in Section
\ref{firstsec}) and $t\in \mathop{\rm dom}(\epsilon)$, let
$Q_{(t,i)}=\{\beta:\Phi^\beta_t(e_t)=i\}$. Let
$Q=\{\beta:\Phi^\beta(\epsilon)\downarrow\}$.

If there is a $g(n,\epsilon)$-good tree for~$Q$ from some string~$\alpha$,
then for some $(t,i)$, there is an $n$-good tree from~$\alpha$ for
$Q_{(t,i)}$.
\end{lem}

\begin{proof}
The number of pairs $(t,i)$ such that $Q_{(t,i)}$ is defined is
$$
a=\sum_{t\in \text{dom}(\epsilon)} h(e_t).
$$
By the assumption that there is a $g(n,\epsilon)$-good tree for~$Q$, it
follows that $a>0$. So since $2^an\ge 2^{a-1}n$, every $2^a n$-good tree is
$2^{a-1}n$-good. Now apply Lemma~\ref{notsharp-basic} to the properties
$Q_{(t,i)}$.
\end{proof}

The following Definition~\ref{claim-basic} will be used in Section
\ref{proof}. We include it here for cross-reference with Lemma~\ref{claim}.

\begin{df}\label{claim-basic}
A tree~$T$ is \emph{$n/m$-good} (read: \emph{$n$-over-$m$ good})
\emph{from~$\alpha$} if there are $m$ many immediate successors of~$\alpha$
which have extensions in~$T$, and for any~$\beta$ having a proper extension
in~$T$, $\beta$ a proper superstring of~$\alpha$, there are $n$ many
immediate successors of~$\beta$ which have extensions in~$T$.
\end{df}

Note that if we imagine trees as growing upwards, this means $T$ is~$n$
``over''~$m$ good in a pictorial sense.

\begin{lem}\label{induct-basic}
Suppose we are given $\alpha$ and~$n$ and a set $P\subseteq \omega^{<\omega}$
such that there is no $n$-good tree from~$\alpha$ for~$P$.

Then if $V$ is an $n$-good tree from~$\alpha$ then there exists $\beta$ such
that
\begin{enumerate}
\item $\beta$ extends an element of~$V$, and \item there is no
$n$-good tree from~$\beta$ for~$P$.
\end{enumerate}
\end{lem}

\begin{proof}
In fact, there exists such $\beta$ which is an element of~$V$, since
otherwise, letting $V_\beta$ be a counterexample for~$\beta$,
$$
V^*=\Union\limits_{\beta \supseteq \alpha,\,\beta \in V} V_\beta
$$
would be $n$-good from~$\alpha$ for~$P$.
\end{proof}

\begin{df}\label{f-basic}
Given a string $\alpha\in\omega^{<\omega}$, $c\in\omega$, and $n\in\omega$,
let $f=f_{\alpha,c,n}$ be defined by the condition: $\Phi_{f(e),t}(x)=i$ if
in $t$ steps a finite tree~$T$ and a number $i<h(e)$ are found such that $T$
is $n$-good from~$\alpha$ for $\{\beta:\Phi_c^{\beta}(e)=i\}$ (and $i$ is the
$i$ occurring for the first such tree found). If such $T$ and $i$ are not
found within $t$ steps, then $\Phi_{f(e),t}(x)\uparrow$.
\end{df}

\begin{df} \label{con-basic} \emph{The Construction.}

At any stage $s+1$, the finite set $D_{s+1}$ will consist of indices $t\le s$
for computations $\Phi^G_t$ that we want to ensure to be divergent. The set
$A_{s+1}$ will consist of what we think of as acceptable strings.

\noindent\emph{Stage 0.}

Let $G[0]=\emptyset$, the empty string, and $\epsilon[0]=\emptyset$. Let
$n[0]=2$. Let $D_0=\emptyset$ and $A_0=\omega^{<\omega}$.

\noindent\emph{Stage $s+1$, $s\ge 0$.}

Let $n[s+1]=g(n[s],\epsilon[s])$, with $g$ as in Definition~\ref{gvec-basic}.

Below we will define $D_{s+1}$. Given $D_{s+1}$, $A_{s+1}$ will be the set of
strings~$\tau$ properly extending $G[s]$ such that for each $t\in D_{s+1}$,
there is no pair $\la T,i\ra$ such that $i<h(e_t)$ and $T$ is a finite
$n[s+1]$-good tree from~$\tau$ for
$Q_{(t,i)}=\{\sigma:\Phi_t^{\sigma}(e_t)\downarrow=i\}$.

Let $e$ be the fixed point of $f=f_{G[s],s,n[s+1]}$ (as defined in Definition
\ref{f-basic}) produced by the Recursion Theorem, i.~e.,
$\Phi_e=\Phi_{f(e)}$.

\noindent\emph{Case 1.} $\Phi_e(e)\downarrow$.

Fix $T$ as in Definition~\ref{f-basic}. Let $D_{s+1}=D_s$. Let $G[s+1]$ be an
extension of~$G[s]$ such that $G[s+1]\in T$ and $G[s+1]\in A_{s+1}$.

\noindent\emph{Case 2.} $\Phi_e(e)\uparrow$. Let
$D_{s+1}=D_s\union\{s\}$. Let $\epsilon[s+1]=\epsilon[s]\union\{(s,e)\}$. In
other words, $e_s=\epsilon(s)$ exists and equals~$e$. Let $G[s+1]$ be any
element of~$A_{s+1}$.

\noindent Let $G=\Union_{s\in\omega} G[s]$.

\noindent\emph{End of Construction.}
\end{df}

\noindent
We now prove that the Construction satisfies Theorem~\ref{theone-basic} in a
sequence of lemmas.

\begin{lem}\label{atleasttwo-basic}
For each $s,t\in\omega$ with $t\le s$, $n_t[s]\ge 2$.
\end{lem}

\begin{proof}
For $s=0$, we have $n[0]=2$. For~$s+1$, we have
$n[s+1]=g(n[s],\epsilon[s])=2^a n[s]$ for a certain $a\ge 0$, by Definition
\ref{g-basic}, hence the lemma follows.
\end{proof}

\begin{lem}\label{accept-basic}
For each $s\ge 0$ the following holds.

\begin{enumerate}
\item[(1)] The Construction at stage~$s$ is well-defined and $G[s]\in A_s$.
In particular, if $s>0$ then in Case 2, $A_s$ is nonempty, and in Case 1,
$A_s$ contains at least one element of~$T$.

\item[(2)] There is no $n[s+1]$-good tree for
$Q=\{\beta:\Phi^\beta(\epsilon[s])\downarrow\}$ from~$G[s]$.

\item[(3)] Every tree~$V$ which is $n[s+1]$-good from~$G[s]$, and
is not just the singleton of~$G[s]$, contains an element of~$A_{s+1}$.
\end{enumerate}
\end{lem}

\begin{proof}
It suffices to show that (1) holds for~$s=0$, and that for each $s\ge 0$, (1)
implies (2) which implies (3), and moreover that (3) for~$s$ implies (1) for
$s+1$.

\noindent (1) holds for~$s=0$ because $G[0]=\emptyset\in\omega^{<\omega}=A_0$.

\noindent\emph{(1) implies (2):}

By definition of~$A_s$ and the fact that $G[s]\in A_s$ by (1) for~$s$, we
have that for each $t\in D_s$, and each $i<h(e_t)$, there is no $n[s]$-good
tree from~$G[s]$ for $Q_{(t,i)}=\{\beta:\Phi^\beta_t(e_t)\downarrow=i\}$.
Hence by Lemma~\ref{g-basic}, there is no $n[s+1]$-good tree for
$Q=\{\beta:\Phi^\beta(\epsilon[s])\downarrow\}$ from~$G[s]$.

\noindent\emph{(2) implies (3):}

Since $V$ is $n[s+1]$-good, by Lemma~\ref{induct-basic} there is an element
$\beta$ of~$V$ from which there is no $n[s+1]$-good tree for~$Q$, and hence
not for any $Q_{(t,i)}$ since $Q_{(t,i)}\subseteq Q$. Moreover, $\beta$
properly extends~$G[s]$, since $V$ is an antichain and is not the singleton
of~$G[s]$. Hence by definition of~$A_{s+1}$, this element $\beta$ belongs to
$A_{s+1}$.

\noindent\emph{(3) for~$s$ implies (1) for~$s+1$:}

If Case 1 holds, let $T$ be the tree found by~$\Phi_e$, i.~e., $T$ is
$n[s+1]$-good from~$G[s]$ (for~$Q_{(s,i)}$ for some~$i$). If $T$ is not just
the singleton of~$G[s]$, and Case 1 holds, then apply (3) for~$s$ to~$T$.

If $T$ is just the singleton of~$G[s]$ or if Case 2 holds, then
apply (3) for~$s$ to any $n[s+1]$-good non-singleton tree
from~$G[s]$. For example, this could be the set of immediate
extensions~$G[s]*k$, $k<n[s+1]$.
\end{proof}

\begin{lem}\label{diverge-basic} For any $s\ge 0$, if $s\in D_{s+1}$ then
$\Phi_s^G(e_s)\uparrow$ or $\Phi_s^G(e_s)\ge h(e_s)$.
\end{lem}

\begin{proof}
Otherwise for some $t\in\omega$, $\Phi_s^{G[t]}(e_s)\downarrow<h(e_s)$. Since
the singleton tree $T=\{G[t]\}$ is $n$-good from~$G[t]$ for all~$n$, hence in
particular $n[t]$-good, this contradicts the fact that by Lemma
\ref{accept-basic}(1), $G[t]\in A_t$.
\end{proof}

\begin{lem}\label{nogood-basic}
There is no $2$-good tree for $\{\beta:\Phi_0^\beta\downarrow\}$ from~$G[0]$.
\end{lem}

\begin{proof}
Suppose a string~$\alpha$ is \textsf{DNR} and $k_1\ne k_2$ are integers. Let
$x=|\alpha|$ (so $x$ is the first input on which $\alpha$ is undefined). It
may or may not be the case that $\varphi_x(x)\downarrow$. In any case, it
cannot be that $k_1=\varphi_x(x)=k_2$. Hence at least one among $\alpha*k_1$
and $\alpha*k_2$ is \textsf{DNR}. This shows that there is no 2-good tree
from~$\alpha$ for the set of non-\textsf{DNR} strings. By Definition
\ref{phizero-basic}, $\Phi_0^\beta\downarrow$ iff $\beta$ is not a
\textsf{DNR} string. As $G[0]=\emptyset$ (the empty string) is a \textsf{DNR}
string, the lemma follows.
\end{proof}

\begin{lem}\label{dee-basic}
$0\in D_1$.
\end{lem}

\begin{proof}
By definition of~$D_1$, it suffices to show that at stage~$1$ of the
Construction, there is no $n[1]$-good tree from~$G[0]$ for
$\{\beta:\Phi^\beta_0\downarrow=i\}$ for any $i<h(e)$. As
$\{\beta:\Phi^\beta_0\downarrow=i\}\subseteq\{\beta:\Phi^\beta_0\downarrow\}$
and $n[1]=2$, this is immediate from Lemma~\ref{nogood-basic}.
\end{proof}

\begin{lem}\label{total-basic}
$G$ is a total function, i.e., $G\in\omega^\omega$.
\end{lem}

\begin{proof}
By Lemma~\ref{accept-basic}(3), $G[s+1]\in A_{s+1}$ for each $s\ge 0$, and
hence by definition of~$A_{s+1}$, $G[s+1]$ is a proper extension of~$G[s]$.
From this the lemma immediately follows.
\end{proof}

\begin{lem}\label{DNR-basic}
$G$ is \textsf{DNR}.
\end{lem}

\begin{proof}
By Lemmas~\ref{diverge-basic} and~\ref{dee-basic}, we have that either
$\Phi^G_0\uparrow$ or $\Phi^G_0\ge h(e_0)$. By Definition~\ref{firstdef},
$h(n)>0$ for all~$n$, whereas by Definition~\ref{phizero-basic},
$\Phi^G_0\downarrow=i$ implies $i=0$. Hence the only possibility is that
$\Phi^G_0\uparrow$. By Definition~\ref{phizero-basic}, this means that $G$ is
\textsf{DNR}.
\end{proof}

\begin{lem}
$G$ computes no $h$-\textsf{DNR} function.
\end{lem}

\begin{proof}
Since each Turing functional has infinitely many indices, it suffices to show
that for each~$s$, $\Phi_s^G$ is not $h$-\textsf{DNR} where $\Phi_s$ is as in
Definition~\ref{phizero-basic}. That is, the fact that we defined our own
$\Phi_0$ is not a problem.

If Case 1 of the construction is followed then
$\Phi^G_s(e)=\Phi^{G[s+1]}(e)=\Phi_e(e)$ because $G[s+1]\in T$. So $\Phi_s^G$
is not $h$-\textsf{DNR}. If Case 2 of the construction is followed then $s\in
D_{s+1}$ and so $\Phi_s^G(e)\uparrow$ or $\Phi_s^G(e)\ge h(e)$ by Lemma
\ref{diverge-basic}. Hence $\Phi_s^G$ is not $h$-\textsf{DNR}.
\end{proof}

\section{The main theorem}\label{proof}

In this section we prove Theorem~\ref{one}, which we restate here.

\begin{thm}\label{theone} For any recursive function $h:\omega\to\omega$,
there exists $G:\omega\to\omega$ (where $G=\oplus_{i\in\omega}G_i$) which is
relatively \textsf{DNR}, and such that for each Turing functional~$\Phi$ and
each $i\in\omega$, $\Phi^{G_0\oplus\cdots\oplus G_i}$ is not an
$h$-\textsf{DNR} function.
\end{thm}

To satisfy the requirement that $G$ be relatively \textsf{DNR}, it will be
convenient to use the following definition.

\begin{df}[Section~\ref{proof} only]\label{phizero}
Let $\Phi_z$, $z\in\omega$ be a sequence of Turing functionals satisfying the
following conditions:

\noindent(1) For all~$z$, $\Phi_z$ queries its oracle on no column
other than columns $0,\ldots,z$. So $\Phi_z^G=\Phi^{G_0\oplus\cdots\oplus
G_z}_z$ for all $G:\omega\to\omega$.

\noindent (2) For all $y\in\omega$,
$\Phi_{2y}^G\downarrow\iff\exists x.G_{2y}(x)=\Phi_x^{G_0\oplus\cdots\oplus
G_{2y-1}}(x)$, and if $\Phi_{2y}^G\downarrow=i\in\omega$ then $i=0$. All
other Turing functionals belong to the set $\{\Phi_{2y+1}:y\in\omega\}$.
\end{df}

In Definition~\ref{phizero}, we note that when $y=0$, $G_0\oplus\cdots\oplus
G_{2y-1}$ equals $\emptyset$. Also $\Phi^G_{2y}$ only queries~$G$ on columns
$0,\ldots,2y$, so (2) is in compliance with (1).

We need the following extension of Definition~\ref{tjo-basic}.

\begin{df}[Good systems of trees]\label{tjo}
Given strings $\sigma_n\in\omega^{<\omega}$, $n\in\omega$, we
define $\sigma=\oplus_{n\in\omega}\sigma_n$ by
$\sigma(2^n(2k+1))=\sigma_n(k)$. We write
$\sigma=\sigma_0\oplus\cdots\oplus\sigma_k$ if
$\sigma_n=\emptyset$ for all $n>k$. Let $\Omega=\omega^{<\omega}$,
and let $\Omega^{<\omega}$ be the set
$$\{\sigma_0\oplus\cdots\sigma_k\: k\in\omega\And\forall i\le k\,.\,\sigma_i\in\omega^{<\omega}\}.$$
Note that $\Omega\subseteq\Omega^{<\omega}$. Conversely,
given~$\sigma\in \Omega^{<\omega}$, the equation
$\sigma(2^n(2k+1))=\sigma_n(k)$ defines each~$\sigma_n$. We refer
to the elements of $\Omega^{<\omega}$ as pseudostrings. For
example, $\la 0,1,1,0\ra\oplus\la 1,1,0\ra$ is pictured as being
defined on initial segments of the first two columns of $\omega$
of length~4 and~3, respectively.

Given $\alpha_0,\ldots,\alpha_x\in\Omega$, $x\ge 0$, we use the
shorthand notation $\vec\alpha_x$ for
$(\alpha_0,\ldots,\alpha_x)$. Similarly for other mathematical
objects: so for example if $n_0,\ldots,n_x$ are integers we
abbreviate $(n_0,\ldots,n_x)$ by~$\vec n_x$. $\vec\alpha_x$ is
also identified with the pseudostring
$\alpha_0\oplus\cdots\oplus\alpha_x$. So given
$\alpha\in\Omega^{<\omega}$, the equation $\alpha=\vec\alpha_x$ is
equivalent to: $\alpha_y=\emptyset$ for all~$y>x$.

If $\vec n_x=(n_0,\ldots,n_x)$ then we can apply operations componentwise,
such as writing $2\vec n_x-1$ for $(2n_0-1,\ldots,2n_x-1)$.

Let $x\ge 0$. A \emph{system of trees} $\vec T=(T_0,\ldots,T_x)=\vec T_x$ is
a tree $T_0$ together with, for each $\sigma_0\in T_0$, a tree
$T_1(\sigma_0)$; and recursively for each $\sigma_k\in
T_k(\vec\sigma_{k-1})$, $0\le k<x$, a tree $T_{k+1}(\vec\sigma_k)$. If
$\sigma_x\in T_x(\vec\sigma_{x-1})$, we say $\vec\sigma_x\in\vec T$. (If
$x=0$, $\vec\sigma_{x-1}$ is the empty sequence and
$T_x(\vec\sigma_{x-1})=T_0$.)

We say that a pseudostring $\beta$ extends a pseudostring $\alpha$
if $\beta(x)=\alpha(x)$ whenever $\alpha(x)$ is defined.

Hence if $\alpha$ and $\beta$ are elements of $\Omega^{x+1}$ for
some $x\ge 0$ then we have a notion of $\beta$ extending $\alpha$.

We call a set $P\subseteq\Omega^{<\omega}$ \emph{open} if for each
$\alpha\in P$ and $\beta$ extending $\alpha$, $\beta\in P$. Given
$x\ge 0$, a subset~$P$ of $\Omega^{x+1}$ is called \emph{open} if
for each $\alpha\in P$ and $\beta$ extending $\alpha$,
$\beta\in\Omega^{x+1}$, we have $\beta\in P$.

Suppose $P$ is a subset of $\Omega^{x+1}$. A system is said to be
a \emph{system for~$P$} if each element of the system is in~$P$.
We write $P(\vec\xi_x)$ to indicate that $\vec\xi_x\in P$; and we
write $P(\vec\xi_{x-1},\cdot)$ for
$\{\xi_x:P(\vec\xi_{x-1},\xi_x)\}$.

A system~$\vec T_x$ is \emph{$\vec n_x$-good from $\vec\sigma_x$} if for each
$\vec\beta_{k-1}\in \vec T_{k-1}$, $0\le k<x$, $T_k(\vec\beta_{k-1})$ is
$n_k$-good from~$\sigma_k$. For~$k=0$ this means that $T_0$ is $n_0$-good
from~$\sigma_0$.

A system $\vec T_x$ is \emph{$(\vec n_{x-1},n_x/m)$-good from~$\vec\sigma_x$}
if

\begin{enumerate}
\item[(1)] $\vec T_{x-1}$ is $\vec n_{x-1}$-good from
$\vec\sigma_{x-1}$, and \item[(2)] for each $\vec\beta_{x-1}\in\vec T_{x-1}$,
$T_x(\vec\beta_{x-1})$ is $n_x/m$-good from~$\sigma_x$ (as in Definition
\ref{claim-basic}).
\end{enumerate}

We say that $\vec\xi_x$ \emph{componentwise extends} $\vec\beta_x$
if for each $0\le y\le x$, $\xi_y$ extends~$\beta_y$; in other
words $\vec\xi_x$ extends $\vec\beta_x$ if we consider them both
as pseudostrings. If $\vec\xi_x$ componentwise extends
$\vec\beta_x$, and $\vec\beta_x$ is an element of a system~$\vec
T_x$, then $\vec\beta_x$ is called the \emph{restriction}
of~$\vec\xi_x$ to~$\vec T_x$. This is well-defined since $\vec
T_x$ is an antichain under the partial order of componentwise
extension.
\end{df}

To prove Theorem~\ref{theone} we will extend the results of Section
\ref{proof-basic} from trees to systems of trees.

\begin{lem}\label{verygoodisgood}
Let $\vec m_x$, $\vec n_x$ be sequences of positive integers such
that $m_i\ge n_i$ for each $0\le i\le x$. Let $\vec T_x$ be a
system of trees. Let $P\subseteq\Omega^{x+1}$, and let
$\vec\sigma_x$ be a sequence of elements of~$\Omega$. If $\vec
T_x$ is $\vec m_x$-good from $\vec\sigma_x$ for~$P$ then $\vec
T_x$ is $\vec n_x$-good from~$\vec\sigma_x$ for~$P$.
\end{lem}

Lemma~\ref{verygoodisgood} is immediate from Definition~\ref{tjo}. The
following is a generalization of Lemma~\ref{kum-basic} to systems of trees.

\begin{lem}\label{kum}
Given $x\ge 0$, a system $\vec T_x$ that is $(2\vec n_x-1)$-good
from some sequence of strings $\vec\sigma_x$, and a subset~$P$ of
$\vec T_x$, there is either an $\vec n_x$-good subset of~$\vec
T_x$ for~$P$ from~$\vec\sigma_x$, or an $\vec n_x$-good subset
of~$\vec T_x$ for the complement of~$P$ from~$\vec\sigma_x$.
\end{lem}

\begin{proof}
The case $x=0$ is Lemma~\ref{kum-basic}. Suppose $x\ge 1$. All sequences
$\vec\alpha_y$, $0\le y\le x$, in the following proof are assumed to be in
$\vec T_y$. Let $\vec\alpha_{-1}$ denote the empty sequence of strings. Call
the elements $\vec\alpha_x$ that are (not) in~$P$ red (blue). So each
$\vec\alpha_x$ is either red or blue.

Inductively, let $y\le x$, $y\ge 0$. Call $\vec\alpha_{y-1}$ red
(blue) if there is an $n_y$-good tree of~$\alpha_y$
from~$\sigma_y$ such that $\vec\alpha_y$ is red (blue). Each
$\vec\alpha_{y-1}$ is either red or blue by Lemma \ref{kum-basic},
since each $\vec\alpha_y$ is either red or blue.

Hence $\vec\alpha_{-1}$ is either red or blue. Say
$\vec\alpha_{-1}$ is red. Then there is an $\vec n_x$-good system
from~$\vec\sigma_x$ for which $\vec\alpha_x$ is red, namely, the
set of all~$\vec\alpha_x$ such that for each $y\le x$,
$\vec\alpha_y$ is red.
\end{proof}

Lemma~\ref{notsharp-basic} generalizes to Lemma~\ref{notsharp} below by the
same proof.

\begin{lem}\label{notsharp}
Let $a\ge 1$ and let $\vec n$ be a finite sequence of positive integers. Let
$\vec T$ be a system of trees which is $2^{a-1}\vec n$-good from some
$\vec\alpha$, and let $P_1,\ldots,P_a$ be sets of sequences of strings such
that $\vec T\subseteq\Union_i P_i$. Then for some~$i$, $\vec T$ has a subset
which is $\vec n$-good for~$P_i$ from~$\vec\alpha$. \qed
\end{lem}

The following definition extends Definition~\ref{gvec-basic}.

\begin{df}\label{gvec}
Given a finite sequence of positive integers $\vec n_x$, $x\ge 0$, and a
finite partial function~$\epsilon$ from~$\omega$ to~$\omega$, let $\vec
g_x(\vec n_x,\epsilon)=2^a\vec n_x$ where
$$
a=\sum_{t\in \text{dom}(\epsilon)} h(\epsilon(t))
$$
and $h$ is as in Section~\ref{firstsec}.
\end{df}

Lemma~\ref{g-basic} now generalizes to the following Lemma~\ref{g}. The proof
of Lemma \ref{g} from Lemma~\ref{notsharp} is identical to the proof of Lemma
\ref{g-basic} from Lemma~\ref{notsharp-basic}.

\begin{lem}\label{g}
Let $\vec n_s$ be a finite sequence of positive integers, let $\epsilon$ be a
finite partial function from~$\omega$ to~$\omega$, and let $\vec g_s$ be the
function defined in Definition~\ref{gvec}.

For each pair $(t,i)$ satisfying $i<h(e_t)$ and $t\in \mathop{\rm
dom}(\epsilon)$, $t\le s$, let
$$
Q_{(t,i)}=\{\vec\beta_s:\Phi^{\vec\beta_t}_t(\epsilon(t))=i\}.
$$
Let
$$
Q=\{\vec\beta_s:\Phi^{\vec\beta_s}(\epsilon)\downarrow\}.
$$
If there is a $\vec g(\vec n,\epsilon)$-good system for~$Q$ from some
$\vec\alpha$, then for some $(t,i)$, there is an $\vec n$-good system from
$\vec\alpha$ for $Q_{(t,i)}$. \qed
\end{lem}

In Lemma~\ref{induct} below we will generalize Lemma~\ref{induct-basic}. To
that end we first prove Lemma~\ref{claim} below. Lemma~\ref{claim} can be
viewed as a generalization of the following observation. Recall the notion of
$b/a$-good from Definition~\ref{claim-basic}. Suppose $a<b$ and there exists
a $b/a$-good tree~$T$ from~$\alpha$ for a set~$P$, but there is no
$b/a+1$-good tree from~$\alpha$ for~$P$. Suppose $k_1,\ldots,k_a$ are $a$
many distinct integers such that for each~$i$, $T$ contains a tree which is
$n$-good from $\alpha*k_i$ for~$P$. Then there is no
$k\not\in\{k_1,\ldots,k_a\}$ such that $T$ contains a tree which is $n$-good
from~$\alpha*k$ for~$P$.

\begin{lem}\label{claim}
Suppose we are given $x\ge 0$, a sequence of strings
$\vec\omega_x$, a sequence of positive integers $\vec n_x$, and an
open set $P\subseteq \Omega^{x+1}$.

Suppose $0\le m<n_x$, and $\vec T_x$ is a $(\vec n_{x-1},n_x/m)$-good system
from~$\vec\omega_x$ for~$P$, but there is no $(\vec n_{x-1},n_x/(m+1))$-good
system from $\vec\omega_x$ for~$P$.

Given $\vec\beta_{x-1}\in\vec T_{x-1}$, let $k_i(\vec\beta_{x-1})$,
$i=1,\ldots,m$ denote $m$ many numbers~$k$ for which $T(\vec\beta_{x-1})$ is
$n_x$-good from~$\omega_x*k$.

Then it is not the case that for every $\vec\beta_{x-1}\in\vec T_{x-1}$ there
exists an $\vec n_{x-1}$-good system $\vec G^{\vec\beta_{x-1}}_{x-1}$ from
$\vec\beta_{x-1}$ for which for each $\vec\xi_{x-1}\in\vec
G^{\vec\beta_{x-1}}_{x-1}$ there exists
$k(\vec\xi_{x-1})\not\in\{k_i(\vec\beta_{x-1}):1\le i\le m\}$ such that there
exists $G^{\vec\beta_{x-1}}_x(\vec\xi_{x-1})$ which is $n_x$-good for
$P(\vec\xi_{x-1},\cdot)$ from $\omega_x*k(\vec\xi_{x-1})$.
\end{lem}

\begin{figure}
\setlength{\unitlength}{4144sp}%
\begingroup\makeatletter\ifx\SetFigFont\undefined
% extract first six characters in \fmtname
\def\x#1#2#3#4#5#6#7\relax{\def\x{#1#2#3#4#5#6}}%
\expandafter\x\fmtname xxxxxx\relax \def\y{splain}%
\ifx\x\y   % LaTeX or SliTeX?
\gdef\SetFigFont#1#2#3{%
  \ifnum #1<17\tiny\else \ifnum #1<20\small\else
  \ifnum #1<24\normalsize\else \ifnum #1<29\large\else
  \ifnum #1<34\Large\else \ifnum #1<41\LARGE\else
     \huge\fi\fi\fi\fi\fi\fi
  \csname #3\endcsname}%
\else \gdef\SetFigFont#1#2#3{\begingroup
  \count@#1\relax \ifnum 25<\count@\count@25\fi
  \def\x{\endgroup\@setsize\SetFigFont{#2pt}}%
  \expandafter\x
    \csname \romannumeral\the\count@ pt\expandafter\endcsname
    \csname @\romannumeral\the\count@ pt\endcsname
  \csname #3\endcsname}%
\fi \fi\endgroup
\begin{picture}(3286,2338)(664,-2564)
\put(1238,-905){\makebox(0,0)[lb]{\smash{\SetFigFont{10}{12.0}{rm}$G_{0}^{\beta_{0}}$}}}
\thinlines \put(2701,-1411){\line(-1, 2){337.600}}
\put(2363,-736){\line( 1, 0){675}}
\put(3038,-736){\line(-1,-2){337.400}} \put(3151,-2311){\line( 1,
2){450}} \put(3151,-2311){\line(-1, 2){450}}
\put(1351,-1411){\line(-1, 2){450}} \put(901,-511){\line( 1,
0){900}} \put(1801,-511){\line(-1,-2){450}}
\put(1126,-2311){\line(-1, 2){450}} \put(676,-1411){\line( 1,
0){900}} \put(1576,-1411){\line(-1,-2){450}}
\put(2476,-1805){\makebox(0,0)[lb]{\smash{\SetFigFont{10}{12.0}{rm}$k_{1}(\beta_{0})$}}}
\put(3376,-904){\makebox(0,0)[lb]{\smash{\SetFigFont{10}{12.0}{rm}$G_{1}^{\beta_{0}}(\xi_{0})$}}}
\put(2532,-904){\makebox(0,0)[lb]{\smash{\SetFigFont{10}{12.0}{rm}$T_{1}(\beta_{0})$}}}
\put(3488,-1804){\makebox(0,0)[lb]{\smash{\SetFigFont{10}{12.0}{rm}$k(\xi_{0})$}}}
\put(3094,-2479){\makebox(0,0)[lb]{\smash{\SetFigFont{10}{12.0}{rm}$\omega_1$}}}
\put(1070,-2479){\makebox(0,0)[lb]{\smash{\SetFigFont{10}{12.0}{rm}$\omega_0$}}}
\put(1070,-1917){\makebox(0,0)[lb]{\smash{\SetFigFont{10}{12.0}{rm}$T_0$}}}
\put(1238,-1523){\makebox(0,0)[lb]{\smash{\SetFigFont{10}{12.0}{rm}$\beta_0$}}}
\put(1688,-454){\makebox(0,0)[lb]{\smash{\SetFigFont{10}{12.0}{rm}$\xi_0$}}}
\put(3601,-1411){\line(-1, 2){337.600}} \put(3263,-736){\line( 1,
0){675}} \put(3938,-736){\line(-1,-2){337.400}}
\end{picture}

\caption{The case $x=1$, $m=1$ of Lemma~\ref{claim}.}
\end{figure}
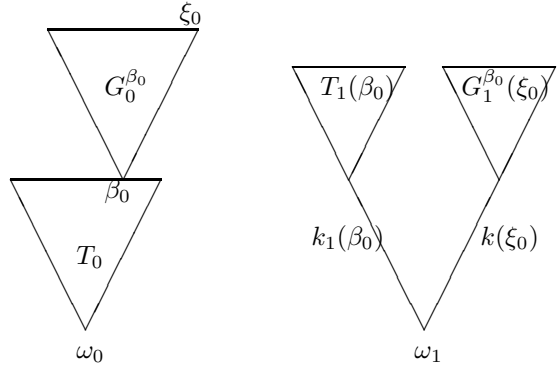

\begin{proof}
Suppose $\vec\xi_{x-1}\in\vec G^{\vec\beta_{x-1}}_{x-1}$. Since $\vec
G^{\vec\beta_{x-1}}_{x-1}$ is good from $\vec\beta_{x-1}$, we know that
$\vec\xi_{x-1}$ extends $\vec\beta_{x-1}$ componentwise. Let
$\vec\beta_{x-1}$ be the restriction of $\vec\xi_{x-1}$ to $\vec T_{x-1}$.

Suppose the lemma fails. Let $\vec G_x=\Union\{\vec
G^{\vec\beta_{x-1}}_x:\vec\beta_{x-1}\in\vec T_{x-1}\}$. Let $\vec H_x=\vec
G_x$ except that
\[
H_x(\vec\xi_{x-1})=G_x(\vec\xi_{x-1})\union
T_x(\vec\beta_{x-1})
\]
for each $\vec\xi_{x-1}$ and its restriction $\vec\beta_{x-1}$ to $\vec
T_{x-1}$.

If $i$ is a number such that $1\le i\le m$, then $T_x(\vec\beta_{x-1})$ is an
$n_x$-good tree for $P(\vec\beta_{x-1},\cdot)$ from $\omega_x*k_i$ and hence
by openness of~$P$ also an $n_x$-good tree for $P(\vec\xi_{x-1},\cdot)$ from
$\omega_x*k_i$ for each $1\le i\le m$, since $\vec\xi_{x-1}$ extends
$\vec\beta_{x-1}$ componentwise.

But $G_x(\vec\xi_{x-1})$ is an $n_x$-good tree for $P(\vec\xi_{x-1},\cdot)$
from $\omega_x*k(\vec\xi_{x-1})$. Hence $H_x(\vec\xi_{x-1})$ is an
$n_x/(m+1)$-good tree for $P(\vec\xi_{x-1},\cdot)$ from~$\omega_x$. So $\vec
H_x$ is an $(\vec n_{x-1},n_x/m+1)$-good system for~$P$ from~$\vec\omega_x$,
contradiction.
\end{proof}

\begin{lem}\label{induct}
Suppose we are given $\vec\alpha_x$ and $\vec n_x$ and an open set
$P\subseteq \Omega^{x+1}$ such that there is no $\vec n_x$-good
system from $\vec\alpha_x$ for~$P$.

If $\vec V_x$ is an $\vec n_x$-good system from $\vec\alpha_x$ then there
exists $\vec\beta_x$ such that
\begin{enumerate}
\item $\beta_x$ extends componentwise an element of~$\vec V_x$,
and

\item there is no $\vec n_x$-good system from $\vec\beta_x$ for
$P$.
\end{enumerate}
\end{lem}

\begin{proof}
By Lemma~\ref{induct-basic}, it is immediate that Lemma~\ref{induct} holds
for $x=0$. Inductively, suppose $x\ge 1$ is given such that Lemma
\ref{induct} holds for $x-1$; we will show that Lemma~\ref{induct} holds for
$x$. From the hypothesis of Lemma~\ref{induct}, we are given that there is no
$\vec n_x$-system from $\vec\alpha_x$ for~$P$, and we let $\vec V_x$ be as in
the statement of Lemma~\ref{induct}.

Let $P_{x-1}$ be the property defined by: for all $\vec\beta_{x-1}$,
$P_{x-1}(\vec\beta_{x-1})$ holds iff there is an $n_x$-good tree from
$\alpha_x$ for the property $\{\alpha:P(\vec\beta_{x-1},\alpha)\}$.

We note that Lemma~\ref{induct} for $x-1$ is applicable to
$\vec\alpha_{x-1}$, $\vec n_{x-1}$, $P_{x-1}$ and $\vec V_{x-1}$. Indeed if
there exists an $\vec n_{x-1}$-good system for $P_{x-1}$ from
$\vec\alpha_{x-1}$ then there would exist an $\vec n_x$-good system for~$P$
from $\vec\alpha_x$, by the definition of the notion of a good system, and
this would contradict the hypothesis of Lemma~\ref{induct} for~$x$. And $\vec
V_{x-1}$ is $\vec n_{x-1}$-good from $\vec\alpha_{x-1}$.

So by Lemma~\ref{induct} for $x-1$, there exists $\vec\sigma_{x-1}$ extending
componentwise an element of $\vec V_{x-1}$, such that there is no $\vec
n_{x-1}$-good system from $\vec\sigma_{x-1}$ for $P_{x-1}$. In other words,
there is no $\vec n_x$-good system from $(\vec\sigma_{x-1},\alpha_x)$ for
$P$.

Fix such a $\vec\sigma_{x-1}$. Let $\vec\gamma_{x-1}$ be the element of $\vec
V_{x-1}$ such that $\vec\sigma_{x-1}$ extends $\vec\gamma_{x-1}$
componentwise, and let $V_x$ be a shorthand for $V_x(\vec\gamma_{x-1})$. Let
$$
Q=\{\vec\tau_x: \text{there is no $\vec n_x$-good system for~$P$ from
     $\vec\tau_x$} \And{}
  \exists\rho\in V_x(\rho\supseteq\tau_x)\}.
$$

To complete the proof of the lemma, we will now construct $\vec\omega_x[0]$,
$\vec\omega_x[1]$, $\ldots$, $\vec\omega_x[p]$ for some $p\in\omega$, such
that $\vec\omega_x[p]$ satisfies the conclusion of Lemma \ref{induct}. To
accomplish this we will ensure that for each $0\le i\le p$, $\omega_x[i]\in
Q$, and $\omega_x[p]\in V_x$.

Let $\vec\omega_x[0]=(\vec\sigma_{x-1},\alpha_x)$. Note $\vec\omega_x[0]\in
Q$. If $\omega_x[0]\in V_x$ then just let~$p=0$.

So suppose we are given $\vec\omega_x[i]\in Q$ for some $i\ge 0$, such that
$\omega_x[i]\not\in V_x$.

Let $m\ge 1$ be maximal such that there is a system $\vec T_x$ which is
$(\vec n_{x-1},n_x/ m)$-good from $\vec\omega_x[i]$ for~$P$, if such an~$m$
exists. If $m$ exists, then since there is no $\vec n_x$-system from
$\vec\omega_x[i]$, we have $m<n_x$; let $\vec T_x$ be such a system.

If $m$ does not exist, let $\vec\omega_x[i+1]$ be
$(\vec\omega_{x-1}[i],\omega_x[i]*k)$ for some~$k$ such that
$\omega_x[i]*k\subseteq\rho$ for some $\rho\in V_x$. Such a~$k$ exists
because $\exists\rho\in V_x.\rho\supseteq\omega_x[i]$ and $\omega_x[i]\not\in
V_x$. Note that $\omega_x[i+1]\in Q$.

So we may assume $m$ does exist. Given $\vec\beta_{x-1}\in\vec T_{x-1}$, we
use the notation $k_i(\vec\beta_{x-1})$, $i=1,\ldots,m$ to list $m$ many
numbers~$k$ for which $T(\vec\beta_{x-1})$ is $n_x$-good from
$\omega_x[i]*k$.

Let us temporarily say that $\vec G_x$ is a useful system for
$\vec\beta_{x-1}$ if $\vec G_{x-1}$ is an $\vec n_{x-1}$-good system from
$\vec\beta_{x-1}$ for which for each $\vec\xi_{x-1}\in\vec G_{x-1}$ there
exists $k(\vec\xi_{x-1})\not\in\{k_i(\vec\beta_{x-1}):1\le i\le m\}$ such
that there exists $G_x(\vec\xi_{x-1})$ which is $n_x$-good for
$P(\vec\xi_{x-1},\cdot)$ from $\omega_x[i]*k(\vec\xi_{x-1})$.

By Lemma~\ref{claim}, it is not the case that for every
$\vec\beta_{x-1}\in\vec T_{x-1}$ there exists a useful system. Thus, let
$\vec\beta_{x-1}$ be a counterexample.

Since $V_x$ is $n_x$-good and $n_x\ge m+1$, $V_x$ is $m+1$-good. We also know
that $\exists\rho\in V_x.\rho\supseteq\omega_x[i]$ and $\omega_x[i]\not\in
V_x$. It follows that there exists $k\not\in\{k_i(\vec\beta_{x-1}):1\le i\le
m\}$ such that $\omega_x[i]*k$ is extended by an element of~$V_x$. Fix such a
$k$ and let $\vec\omega_x[i+1]=(\vec\beta_{x-1},\omega_x[i]*k)$.

If there existed an $\vec n_x$-good system $\vec D_x$ for~$P$ from
$(\vec\beta_{x-1},\omega_x[i]*k)$, then $\vec D_x$ would be a useful system
for $\vec\beta_{x-1}$ (with $k(\vec\xi_{x-1}):=k$ for each $\vec\xi_{x-1}$),
contradiction. Hence $\vec\omega_x[i+1]\in Q$.

Since $V_x$ is finite, we eventually reach an~$i$ such that $\omega_x[i]\in
V_x$. Letting $p=i$ completes the proof of the lemma.
\end{proof}

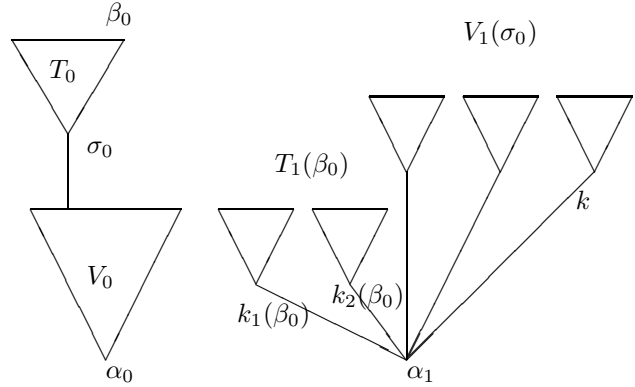
\begin{figure}
\setlength{\unitlength}{4144sp}%
\begingroup\makeatletter\ifx\SetFigFont\undefined
% extract first six characters in \fmtname
\def\x#1#2#3#4#5#6#7\relax{\def\x{#1#2#3#4#5#6}}%
\expandafter\x\fmtname xxxxxx\relax \def\y{splain}%
\ifx\x\y   % LaTeX or SliTeX?
\gdef\SetFigFont#1#2#3{%
  \ifnum #1<17\tiny\else \ifnum #1<20\small\else
  \ifnum #1<24\normalsize\else \ifnum #1<29\large\else
  \ifnum #1<34\Large\else \ifnum #1<41\LARGE\else
     \huge\fi\fi\fi\fi\fi\fi
  \csname #3\endcsname}%
\else \gdef\SetFigFont#1#2#3{\begingroup
  \count@#1\relax \ifnum 25<\count@\count@25\fi
  \def\x{\endgroup\@setsize\SetFigFont{#2pt}}%
  \expandafter\x
    \csname \romannumeral\the\count@ pt\expandafter\endcsname
    \csname @\romannumeral\the\count@ pt\endcsname
  \csname #3\endcsname}%
\fi \fi\endgroup
\begin{picture}(3737,2454)(1226,-5659)
\put(1801,-3436){\makebox(0,0)[lb]{\smash{\SetFigFont{10}{12.0}{rm}$\beta_0$}}}
\thinlines \put(3263,-5011){\line( 3,-4){337.680}}
\put(4726,-4336){\line(-1,-1){1125}}
\put(4163,-4336){\line(-1,-2){562.400}} \put(3601,-4336){\line(
0,-1){1125}} \put(2700,-5011){\line(-1, 2){225}}
\put(2475,-4561){\line( 1, 0){450}}
\put(2925,-4561){\line(-1,-2){225}} \put(3263,-5011){\line(-1,
2){225}} \put(3038,-4561){\line( 1, 0){450}}
\put(3488,-4561){\line(-1,-2){225}} \put(3601,-4336){\line(-1,
2){225}} \put(3376,-3886){\line( 1, 0){450}}
\put(3826,-3886){\line(-1,-2){225}} \put(4163,-4336){\line(-1,
2){225}} \put(3938,-3886){\line( 1, 0){450}}
\put(4388,-3886){\line(-1,-2){225}} \put(4726,-4336){\line(-1,
2){225}} \put(4501,-3886){\line( 1, 0){450}}
\put(4951,-3886){\line(-1,-2){225}} \put(1576,-4561){\line( 0,
1){450}} \put(1801,-5461){\line(-1, 2){450}}
\put(1351,-4561){\line( 1, 0){900}}
\put(2251,-4561){\line(-1,-2){450}} \put(1576,-4110){\line(-3,
5){337.412}} \put(1238,-3548){\line( 1, 0){675}}
\put(1913,-3548){\line(-3,-5){337.147}}
\put(3601,-5573){\makebox(0,0)[lb]{\smash{\SetFigFont{10}{12.0}{rm}$\alpha_1$}}}
\put(3151,-5123){\makebox(0,0)[lb]{\smash{\SetFigFont{10}{12.0}{rm}$k_{2}(\beta_{0})$}}}
\put(2588,-5236){\makebox(0,0)[lb]{\smash{\SetFigFont{10}{12.0}{rm}$k_{1}(\beta_{0})$}}}
\put(2813,-4336){\makebox(0,0)[lb]{\smash{\SetFigFont{10}{12.0}{rm}$T_{1}(\beta_{0})$}}}
\put(4613,-4561){\makebox(0,0)[lb]{\smash{\SetFigFont{10}{12.0}{rm}$k$}}}
\put(3938,-3548){\makebox(0,0)[lb]{\smash{\SetFigFont{10}{12.0}{rm}$V_{1}(\sigma_{0})$}}}
\put(1801,-5573){\makebox(0,0)[lb]{\smash{\SetFigFont{10}{12.0}{rm}$\alpha_0$}}}
\put(1688,-5011){\makebox(0,0)[lb]{\smash{\SetFigFont{10}{12.0}{rm}$V_0$}}}
\put(1688,-4223){\makebox(0,0)[lb]{\smash{\SetFigFont{10}{12.0}{rm}$\sigma_0$}}}
\put(1463,-3773){\makebox(0,0)[lb]{\smash{\SetFigFont{10}{12.0}{rm}$T_0$}}}
\put(2701,-5011){\line( 2,-1){900}}
\end{picture}

\caption{The case $x=1$, $m=2$ of Lemma~\ref{induct}.}
\end{figure}

The following definition extends Definition~\ref{f-basic} to systems of
trees.

\begin{df}\label{f}
Given $x\ge 0$, a sequence of strings $\vec\alpha_x$ where each
$\alpha_i\in\Omega$, $c\in\omega$ and a sequence of positive
integers $\vec n_x$, let $f=f_{\vec\alpha,c,\vec n_x}$ be defined
by the condition: for all $z,t\in\omega$, $\Phi_{f(e),t}(z)=i$ if
in $t$ steps a finite system of trees $\vec T_x$ and a number
$i<h(e)$ are found such that $\vec T_x$ is $\vec n_x$-good from
$\vec\alpha_x$ for $\{\vec\beta_x:\Phi_c^{\vec\beta_x}(e)=i\}$
(and $i$ is the~$i$ occurring for the first such tree found). If
no such $\vec T_x$ and $i$ are found within $t$ steps, then
$\Phi_{f(e),t}(x)$ is undefined.
\end{df}

\begin{df} \label{con} \emph{The Construction.}
At any stage $s+1$, the finite set $D_{s+1}$ will consist of
indices $t\le s$ for computations $\Phi^G_t$ that we want to
ensure are divergent. The set $A_{s+1}$ will consist of what we
think of as acceptable pseudostrings. At stage $s$ we will define
a sequence of positive integers $\vec n[s]=\vec
n_s[s]=(n_0[s],\ldots,n_s[s])$; so the entries of this vector are
$n_t[s]$, $0\le t\le s$.

\noindent\emph{Stage 0.}

Let $G[0]=\emptyset$, the empty pseudostring, and
$\epsilon[0]=\emptyset$. Let $\vec n[0]=\la 2\ra$. Let
$D_0=\emptyset$ and $A_0=\Omega$.

\noindent\emph{Stage $s+1$, $s\ge 0$.}

Below we will define $D_{s+1}$. Given $D_{s+1}$, $A_{s+1}$ will be
the set of pseudostrings $\tau=\vec\tau_s$ such that $\tau_t$
properly extends $G_t[s]$ for each $t\le s$, and for each $t\in
D_{s+1}$, there is no pair $\la
%(note $G_t[s]$ is nonempty on columns $t\le s-1$ only)
\vec T_t,i\ra$ such that $i<h(e_t)$ and $\vec T_t$ is a finite $\vec
n[t]$-good tree from~$\tau$ for
$Q_{(t,i)}=\{\sigma:\Phi_t^{\sigma}(e_t)=i\}$.

Let $\vec n_{s+1}[s+1]=(\vec g_s(\vec n_s[s],\epsilon[s]),2)$, with $\vec
g_s$ as in Definition~\ref{gvec}.

Let $e$ be the fixed point of $f=f_{G[s],s,\vec g(\vec n[s],\epsilon[s])}$
(as in Definition~\ref{f}) produced by the Recursion Theorem, i.~e.,
$\Phi_e=\Phi_{f(e)}$.

\noindent\emph{Case 1.} $\Phi_e(e)\downarrow$.

Fix $\vec T_{s+1}$ as in Definition~\ref{f}. Let $D_{s+1}=D_s$. Let $G[s+1]$
be an extension (columnwise, nonempty on columns $\le s$ only) of~$G[s]$ such
that $G[s+1]\in \vec T_{s+1}$ and $G[s+1]\in A_{s+1}$.

\noindent\emph{Case 2.} $\Phi_e(e)\uparrow$. Let
$D_{s+1}=D_s\union\{s\}$. Let $\epsilon[s+1]:=\epsilon[s]\union\{(s,e)\}$, so
$e_s:=e$. Let $G[s+1]$ be any element of~$A_{s+1}$.

\noindent Let $G=\Union_{s\in\omega} G[s]$.

\noindent\emph{End of Construction.}
\end{df}

\noindent

We now prove that the Construction satisfies Theorem~\ref{theone} in a
sequence of lemmas.

\begin{lem}\label{atleasttwo}
For each $s,t\in\omega$ with $t\le s$, $n_t[s]\ge 2$.
\end{lem}

\begin{proof}
For $s=0$, we have $\vec n[0]=(n_0[0])=(2)$.

For $s+1$, we have $\vec n[s+1]=(\vec g_s(\vec n[s],\epsilon[s]),2)$ and
$g_s(\vec n[s],\epsilon[s])=2^a\vec n[s]$ for a certain $a\ge 0$, by
Definition~\ref{g}, hence the lemma follows.
\end{proof}

Note that $G[s]$, while only nonempty on columns $\le s-1$, can be considered
as defined on all columns, or as many additional columns as desired, in
accordance with Definition~\ref{firstdef}. For example, in Lemma
\ref{accept}(3) we think of~$G[s]$ as
$$
G_0[s]\oplus\cdots\oplus G_{s-1}[s]\oplus G_s[s]
$$
with $G_s[s]=\emptyset$.

\begin{lem}\label{accept}
For each $s \ge 0$, the following holds.

\begin{enumerate}
\item[(1)] The Construction at stage~$s$ is well-defined and
$G[s]\in A_s$. In particular, if $s>0$ then if Case 2 applies then $A_s$ is
nonempty, and if Case 1 applies then $A_s$ contains elements of~$\vec T$.

\item[(2)] There is no $\vec g_s(\vec n_s[s],\epsilon[s])$-good
system of trees for
$$
Q=\{\vec\beta_s:\Phi^{\vec\beta_s}(\epsilon[s])\downarrow\}
$$
from~$G[s]$.

\item[(3)] Every system~$\vec V_s$ which is $\vec g_s(\vec n_s[s],
\epsilon[s])$-good from~$G[s]$, and is not just the singleton
of~$G[s]$, contains an element of~$A_{s+1}$.
\end{enumerate}
\end{lem}

\begin{proof}
It suffices to show that (1) holds for~$s=0$, and that for each $s\ge 0$, (1)
implies (2) which implies (3), and moreover that (3) for~$s$ implies (1) for
$s+1$.

\noindent (1) holds for~$s=0$ because
$G[0]=\emptyset\in\Omega=A_0$.

\noindent\emph{(1) implies (2):}

Suppose $\vec U_s$ is a $\vec g_s(\vec n_s[s],\epsilon[s])$-good
system for $Q$ from~$G[s]$. As each~$\Phi_t$ only queries columns
$\le t$, and $t\in D_s=\text{dom}(\epsilon[s])$ implies $t<s$, we
see that each $\Phi_t$ for $t\in D_s$ only queries columns $\le
s-1$, so $\Phi^X(\epsilon[s])$ only queries columns $\le s-1$ for
any~$X$, and in particular only queries columns $\le s$. By
Lemma~\ref{g}, there is an $\vec n[s]=\vec n_s[s]$-good
system~$\vec V_s$ for
$$
Q_{(t,i)}=\{\vec\beta_s:\Phi^{\vec\beta_t}_t(e_t)\downarrow=i\}
$$
for some $t\in D_s$ and $i<h(e_t)$ from~$G[s]$.

Now $\vec n[s]=(\vec g_{s-1}(\vec n[s-1],\epsilon[s-1]),2)$, hence the
restriction $\vec V_{s-1}$ is $\vec g_{s-1}(\vec n[s-1],\epsilon[s-1])$-good.

For each $t^*\le s-1$, $g_{t^*}(\vec n[s-1],\epsilon[s-1])=2^a
n_{t^*}[s-1]\ge n_{t^*}[s-1]$ (for a certain $a\ge 0$). Applying this to
$t^*\le t$ (since $t\le s-1$), by Lemma~\ref{verygoodisgood}, the further
restriction $\vec V_t$ is $\vec n[t]$-good.

By (1) for~$s$, $G[s]\in A_s$. Recall that $A_s$ is the set of
pseudostrings $\tau=\vec\tau_{s-1}$ such that $\tau_{t^*}$
properly extends $G_{t^*}[s-1]$ for each $t^*\le s-1$, and for
each $t^*\in D_s$ (hence $t^*\le s-1$), there is no pair $\la\vec
T_{t^*},i^*\ra$ such that $i^*<h(e_{t^*})$ and $\vec T_{t^*}$ is a
finite $\vec n[t^*]$-good tree from~$\tau$ for
$Q_{(t^*,i^*)}=\{\vec\sigma_{t^*}:\Phi_{t^*}^{\vec\sigma_{t^*}}(e_{t^*})=i^*\}$.

Applying this with $t^*:=t$ and $i^*:=i$, we have that $G[s]=G[s]_{s-1}$ and
there is no pair $\la\vec T_t,i\ra$ such that $i<h(e_t)$ and $\vec T_t$ is a
finite $\vec n[t]$-good tree from~$G[s]$ for
$Q_{(t,i)}=\{\vec\sigma_t:\Phi_t^{\vec\sigma_t}(e_t)=i\}$.

But $\vec V_t$ is exactly such a tree $\vec T_t$, so we have a contradiction.

\noindent\emph{(2) implies (3):}

Since $\vec V_s$ is $\vec g_s(\vec n_s[s],\epsilon[s])$-good, by
Lemma \ref{induct-basic} there is an element $\vec\beta_s$
of~$\vec V_s$ from which there is no $\vec g_s(\vec
n_s[s],\epsilon[s])$-good tree for~$Q$, and hence not for any
$Q_{(t,i)}$ since $Q_{(t,i)}\subseteq Q$. Moreover $\vec\beta_s$
properly extends~$G[s]$, since $\vec V_s$ is not just the
singleton of~$G[s]$. So as $\vec V_s$ is $\vec g_s(\vec
n_s[s],\epsilon[s])$-good, as $\vec n[s+1]=\vec n_{s+1}[s+1]=(\vec
g_s(\vec n_s[s],\epsilon[s]),2)$ and as by Lemma~\ref{atleasttwo},
$n_t[s]\ge 2$ for each $t\le s$, it follows that every column
$\beta_t$ of $\vec\beta_s$ extends $G_t[s]$ properly.

Hence by definition of $A_{s+1}$, this element $\vec\beta$ belongs to
$A_{s+1}$.

\noindent\emph{(3) for $s$ implies (1) for $s+1$:}

If Case 1 obtains, let $\vec T_s$ be the tree found by~$\Phi_e$,
i.~e., $\vec T_s$ is $\vec g_s(\vec n_s[s],\epsilon[s])$-good
from~$G[s]$ (for $Q_{(s,i)}$ for some~$i$). If $\vec T_s$ is not
just the singleton of~$G[s]$, and Case 1 obtains, then apply (3)
for~$s$ to~$\vec T_s$.

If $\vec T_s$ is just the singleton of~$G[s]$ or if Case 2
obtains, then apply (3) for~$s$ to any~$\vec g_s(\vec
n_s[s],\epsilon[s])$-good non-singleton system of trees
from~$G[s]$.
\end{proof}

\begin{lem}\label{diverge}
For any~$s\ge 0$, if $s\in D_{s+1}$ then $\Phi_s^G(e_s)\uparrow$ or
$\Phi_s^G(e_s)\ge h(e_s)$.
\end{lem}

\begin{proof}
Otherwise for some $t\in\omega$, $\Phi_s^{G[t]}(e_s)\downarrow<h(e_s)$. Since
the system whose only element is $G[t]$ is $\vec n_x$-good from~$G[t]$ for
all~$\vec n_x$ with $x\ge t$, hence in particular $\vec n[t]=\vec
n_t[t]$-good, this contradicts the fact that $G[t]\in A_t$.
\end{proof}

For each $x\in\omega$, let $\vec 1_x=(1_0,\ldots,1_x)$ be the sequence of
length $x+1$ consisting of all 1's, i.~e., where $1_0=\cdots=1_x=1$.

\begin{lem}\label{nogood}
For each $y\ge 0$, there is no $(\vec 1_{2y-1},2)$-good system from~$G[2y]$
for the property $\{\beta:\Phi^\beta_{2y}\downarrow\}$.
\end{lem}

\begin{proof}
Suppose there is such a system $\vec T_{2y}$.

First suppose $\vec T_{2y}$ has only one element. Then this element is
$G[2y]$, by the definition of a good system from~$G[2y]$. Hence
$\Phi^{G[2y]}_{2y}\downarrow$. But $G_{2y}[2y]$, column $2y$ of~$G$ as
constructed during stage~$2y$, is empty. So by Definition~\ref{phizero},
$\Phi^{G[2y]}_{2y}\uparrow$, so we have a contradiction.

Now suppose $\vec T_{2y}$ has more than one element. Given
$G_0\oplus\cdots\oplus G_{2y-1}$, there is at most one value of $G_{2y}(0)$
such that $\la G_{2y}(0)\ra$ is not a \textsf{DNR}$^{G_0\oplus\cdots \oplus
G_{2y-1}}$ string. Hence for any sequence of positive integers $\vec n_{2y}$,
if $\vec T_{2y}$ is $\vec n_{2y}$-good from~$G[s]$ then $n_{2y}\le 1$, so
$2\le 1$, which is a contradiction.
\end{proof}

\begin{lem}\label{dee}
For each $y\in\omega$, $2y\in D_{2y+1}$.
\end{lem}
\begin{proof}
By definition of $D_{2y+1}$, it suffices to show that at stage $2y+1$ of the
Construction, there is no $\vec g(\vec n[2y],\epsilon[2y])$-good system from
$G[2y]$ for $\{\beta:\Phi^\beta_{2y}\downarrow=i\}$ for any $i<h(e)$. We will
show this in fact for $\{\beta:\Phi^\beta_{2y}\downarrow\}$.

Suppose there is such a system $\vec T_{2y}$. By Lemma~\ref{atleasttwo},
$g_x(\vec n[2y],\epsilon[2y])\ge 2$ for $0\le x\le 2y$. Hence by Lemma
\ref{verygoodisgood}, $\vec T_{2y}$ is $(\vec 1_{2y-1},2)$-good. By Lemma
\ref{nogood}, we have a contradiction.
\end{proof}

\begin{lem}\label{total} $G$ is a total function, i.~e.,
$G\in\omega^\omega$.
\end{lem}
\begin{proof}
By Lemma~\ref{accept}(3), $G[s+1]\in A_{s+1}$ for each $s\ge 0$, and hence by
definition of $A_{s+1}$, $G_t[s+1]$ is a proper extension of~$G_t[s]$ for
each $t\le s$. From this the lemma immediately follows.
\end{proof}

\begin{lem}\label{DNR}
$G$ is relatively \textsf{DNR}.
\end{lem}

The proof of Lemma~\ref{DNR} from Definition~\ref{phizero}, Lemma
\ref{diverge} and Lemma~\ref{dee} is formally identical to the
proof of Lemma \ref{DNR-basic} from
Definition~\ref{phizero-basic}, Lemma \ref{diverge-basic} and
Lemma~\ref{dee-basic}.

\begin{lem} For each $y\in\omega$, $G_0\oplus\cdots\oplus G_y$
computes no $h$-\textsf{DNR} function.
\end{lem}
\begin{proof}
It suffices to show that given $y\in\omega$, and a Turing
functional~$\Psi$ which does not query its oracle beyond column
$2y+1$, $\Psi^{G_0\oplus\cdots\oplus G_{2y+1}}$ is not
$h$-\textsf{DNR}. In the Construction we have been considering the
Turing functionals~$\Phi_z, z\in\omega$ of
Definition~\ref{phizero}. Since each Turing functional has
infinitely many indices, it follows from Definition~\ref{phizero}
that there are infinitely many odd numbers $s\ge 2y+1$ such that
$$\Psi^{G_0\oplus\cdots\oplus
G_{2y+1}}=\Phi_s^{G_0\oplus\cdots\oplus
G_{2y+1}}=\Phi_s^{G_0\oplus\cdots\oplus G_s}=\Phi_s^G.$$ Fix such
an~$s$ and consider stage $s+1$ of the Construction. If Case 1
holds then $\Phi^G_s(e)=\Phi_e(e)$ and so $\Phi_s^G$ is not
$h$-\textsf{DNR}. If Case 2 holds then by Lemma~\ref{diverge},
$\Phi_s^G(e)\uparrow$ or $\Phi_s^G(e)\ge h(e)$. Hence $\Phi_s^G$
is not $h$-\textsf{DNR}.
\end{proof}
\bibliographystyle{asl}
\bibliography{dnr}
\end{document}